\documentclass[12pt]{article}
\usepackage{amssymb, amsmath, amsthm, amsfonts,xcolor, enumerate}
\usepackage{tikz}
\usetikzlibrary{patterns}
\usetikzlibrary{shapes}
\usepackage{fullpage}
\usepackage{hyperref}
\newtheorem{theorem}{Theorem}[section]
\newtheorem{lemma}[theorem]{Lemma}
\newtheorem{cor}[theorem]{Corollary}
\theoremstyle{definition}

\theoremstyle{plain}
\newtheorem{claim}[theorem]{Claim}

\theoremstyle{definition}

\theoremstyle{definition}

\theoremstyle{definition}

\theoremstyle{definition}

\theoremstyle{definition}

\theoremstyle{definition}

\theoremstyle{definition}

\newcommand{\ep}{\varepsilon}

\newcommand{\de}{\delta}

\newcommand{\De}{\Delta}

\newcommand{\cH}{\mathcal{H}}

\newcommand{\bP}{\ensuremath{\mathbb{P}}}
\newcommand{\bE}{\ensuremath{\mathbb{E}}}

\title{Subdivisions of a large clique in $C_6$-free graphs}
 \author{
 J\'ozsef Balogh
 \thanks{Department of Mathematical Sciences,
 University of Illinois at Urbana-Champaign, Urbana, Illinois 61801, USA {\tt
jobal@math.uiuc.edu}. Research is partially supported by Simons Fellowship, NSF
CAREER Grant
 DMS-0745185, Marie Curie FP7-PEOPLE-2012-IIF 327763 and Arnold O.
 Beckman Research Award (UIUC Campus Research Board 13039).}
 \quad\quad
 Hong Liu
 \thanks{Department of Mathematical Sciences,
 University of Illinois at Urbana-Champaign, Urbana, Illinois 61801, USA {\tt
hliu36@illinois.edu}.}
 \quad\quad
 Maryam Sharifzadeh
 \thanks{Department of Mathematical Sciences,
 University of Illinois at Urbana-Champaign, Urbana, Illinois 61801, USA {\tt
sharifz2@illinois.edu}.}
 }

\begin{document}
\maketitle

\begin{abstract}
Mader conjectured that every $C_4$-free graph has a subdivision of a clique of order
linear in its average degree. We show that every $C_6$-free graph has such a subdivision of a large clique.

We also prove the dense case of Mader's conjecture in a stronger sense, i.e.,~for
every $c$, there is a $c'$ such that every $C_4$-free graph with average degree
$cn^{1/2}$ has a subdivision of a clique $K_\ell$ with $\ell=\lfloor c'n^{1/2}\rfloor$ where
every edge is subdivided exactly $3$ times.
\end{abstract}

\section{Introduction}
 A \emph{subdivision} of a clique $K_\ell$, denoted by $TK_{\ell}$, is a graph obtained
from $K_{\ell}$ by subdividing each of its edges into internally vertex-disjoint
paths. Bollob\'as and Thomason~\cite{B-Th}, and
independently Koml\'os and Szemer\'edi~\cite{K-Sz-2} proved the following celebrated result.

\begin{theorem}\label{thm-B-T-K-Sz}
Every graph of average degree $d$ contains a subdivision of a clique of order
$\Omega(\sqrt{d})$. 
\end{theorem}

Theorem~\ref{thm-B-T-K-Sz} is best possible: the disjoint union of $K_{d,d}$'s contains no subdivision of $K_{\ell}$ with $\ell\ge\sqrt{8d}$ (observed first by Jung~\cite{J}).
 
Mader~\cite{M} conjectured that if a graph is $C_4$-free,
then one can find a subdivision of a much larger clique, of order linear in its
average degree. Two major steps towards this conjecture were made by K\"uhn and Osthus: in~\cite{K-O-1},
they showed that if the graph $G$ has girth at least 15 and large average degree, then the conjecture is true in a stronger sense: a subdivision of $K_{\de(G)+1}$ is guaranteed;
in~\cite{K-O-2}, they showed that one can find a subdivision of a clique of order
almost linear, $\Omega(d/\log^{12}d)$, in any $C_4$-free graph with average
degree $d$.

Extending ideas in~\cite{K-Sz-1} and~\cite{K-Sz-2}, we prove that every $C_6$-free graph
has such a subdivision of a large clique.
\begin{theorem}\label{mainc6}
Let $G$ be a $C_6$-free graph with average degree $d$. Then a $TK_{\ell}$ is a
subgraph of $G$ with $\ell=\lfloor cd\rfloor$ for some small positive constant $c$
independent of $d$.
\end{theorem}

Similar proof gives the following result, whose proof is omitted.

\begin{theorem}\label{mainc2k}
Let $G$ be a $C_{2k}$-free graph with $k\ge 3$ and average degree $d$. Then a
$TK_{\ell}$ is a subgraph of $G$ with $\ell=\lfloor cd\rfloor$ for some small positive 
constant $c$ independent of $d$.
\end{theorem}

It is known that any $C_4$-free $n$-vertex graph has at most $O(n^{3/2})$ edges (see~\cite{K-S-T}). 
Our next result verifies the dense case of Mader's conjecture in a stronger
sense.

\begin{theorem}\label{c4dense}
For every $c>0$ there is a $c'>0$ such that  the following holds. Let $G$ be a
$C_4$-free $n$-vertex graph with $cn^{3/2}$ edges. Then $G$ contains
a $TK_{\ell}$ with $\ell=\lfloor c'n^{1/2}\rfloor$, in which every edge of
the $K_\ell$ is subdivided exactly $3$ times.
\end{theorem}

Theorem~\ref{c4dense} can also be viewed as an extension of the following result of Alon, Krivelevich and Sudakov~\cite{A-K-S} for $C_4$-free graphs. Settling a question of Erd\H{o}s~\cite{E}, they showed, using the dependent random choice lemma, that if the average degree of a graph is of order $\Omega(n)$, then there is a $TK_{\ell}$ with $\ell=\Omega(n^{1/2})$, in which every edge of the $K_\ell$ is subdivided exactly once.

\medskip

\noindent\textbf{Notation:} For a vertex $v$, denote by $S(v,i)$ the $i$-th
sphere around $v$, i.e.,~the set of vertices of distance $i$ from $v$ and denote
by $B(v,r)$ the ball of vertices of radius $r$ around $v$, so $B(v,r)=\cup_{i\le
r}S(v,i)$. For a set $X\subseteq V(G)$, denote by $\Gamma(X)$ the external
neighborhood of $X$, that is $\Gamma(X):=N(X)\setminus X$. Denote by $d(G)$ the
average degree of $G$ and for $S\subseteq V(G)$ denote by $d(S)$ the average
degree of the induced subgraph $G[S]$. For a set of vertices $S$, denote by
$N_i(S)$ the $i$-th common neighborhood of $S$, i.e.,~vertices of distance
exactly $i$ from every vertex in $S$. For a set $B\subseteq V(G)$, let
$\Delta(B):=\max_{v\in B}d_G(v)$ and $\delta(B):=\min_{v\in B}d_G(v)$.

\medskip

We will omit floors and ceilings signs when they are not crucial.

\section{Preliminaries}
For any graph $G$, there is a bipartite subgraph $G'$ such that $e(G')\ge
e(G)/2$. We shall use a result of Gy\"ori~\cite{G} which states that every bipartite $C_6$-free graph has a $C_4$-free subgraph with at least half of its edges. So having a loss of factor of $4$ in the average degree, we may assume that our $C_6$-free graph is
bipartite and also $C_4$-free. Following Koml\'os and Szemer\'edi~\cite{K-Sz-1}, we introduce the following concept.

\medskip

\noindent\textbf{$(\ep_1,t)$-expander:} For $\ep_1>0$ and $t>0$, let $\ep(x)$ be
the function as follows:
\begin{eqnarray}\label{epsilon}
\ep(x)=\ep(x,\ep_1,t):=\left\{\begin{tabular}{ l c r }
 $0$ & \mbox{ if } $x<t/5$ \\
 $\ep_1/\log^2(15x/t)$ & $\mbox{ if } x\ge t/5$. \\
\end{tabular}
\right.
\end{eqnarray}
For the sake of brevity, on $\ep(x)$ we do not write the dependency of $\ep_1$ and $t$ when it is clear from the context. Note that $\ep(x)\cdot x$ is increasing for $x\ge t/2$. A graph $G$ is an \emph{$(\ep_1,t)$-expander} if $|\Gamma(X)|\ge \ep(|X|)\cdot |X|$ for all subsets $X\subseteq V$ of size $t/2\le |X|\le |V|/2$.

Koml\'{o}s and Szemer\'{e}di~\cite{K-Sz-1,K-Sz-2} showed that
every graph $G$ contains an $(\ep,t)$-expander that is almost as dense as $G$. 
\begin{theorem}\label{k-sz-expander-original}
Let $t>0$, and choose $\ep_1>0$ sufficiently small (independent of $t$) so that
$\ep=\ep(x)$ defined in~\eqref{epsilon} satisfies $\int_1^{\infty}\frac{\ep(x)}{x} dx<\frac{1}{8}$. Then every graph $G$ has a subgraph $H$ with $d(H)\ge d(G)/2$ and $\de(H)\ge d(H)/2$, which is an $(\ep_1,t)$-expander. 
\end{theorem}

\noindent\textbf{Remark:} The subgraph $H$ might be much smaller than $G$. For
example if $G$ is a vertex-disjoint collection of $K_{d+1}$'s, then $H$ will
be just one of the $K_{d+1}$'s. 

\vspace{2mm}

We will use the following version of Theorem~\ref{k-sz-expander-original}.
\begin{cor}\label{k-sz-expander}
There exists $\ep_0$ with $0<\ep_0<1$ such that for every $0<\ep_1\le\ep_0$, $\ep_2>0$ and every graph $G$, there is a subgraph $H\subseteq G$ with $d(H)\ge d(G)/2$ and $\de(H)\ge d(H)/2$ which is an $(\ep_1,\ep_2d(H)^2)$-expander.
\end{cor}
\begin{proof}
Let $G'\subseteq G$ be a subgraph maximizing $d(G')$ and define $t':=\ep_2d(G')^2/4$. If $\ep_0$ is sufficiently small, then for any $\ep_1\le \ep_0$, applying Theorem~\ref{k-sz-expander-original} yields a $(4\ep_1, t')$-expander $H\subseteq G'$ with $d(G')/2\le d(H)\le d(G')$ and $\de(H)\ge d(H)/2$. Define $t:=\ep_2d(H)^2$. Since $d(G')/2\le d(H)\le d(G')$, we have $t'\le t\le 4t'$. A simple calculation shows that for every $x\ge t/2$, 
$$\frac{4\ep_1}{\log^2(15x/t')}\ge \frac{\ep_1}{\log^2(15x/t)}.$$
Hence $H$ is an $(\ep_1, t)$-expander as desired.
\end{proof}
Every $(\ep_1,t)$-expander graph has the following robust ``small diameter'' property (see Corollary 2.3 in~\cite{K-Sz-2}):
\begin{cor}\label{diameter}
If $G$ is an $(\ep_1,t)$-expander, then any two vertex sets, each of size at least
$x\ge t$, are of distance at most $$diam:=diam(n,\ep_1,t)=\frac{2}{\ep_1}\log^3(15n/t),$$
and this remains true even after deleting $x\ep(x)/4$ arbitrary vertices from $G$.
\end{cor}

By Corollary~\ref{k-sz-expander}, we may assume, when proving
Theorem~\ref{mainc6}, that $G$ is a bipartite, $\{C_4,C_6\}$-free,
$(\ep_1,t)$-expander graph with average degree $d$, $\de(G)\ge d/2$ and $t=\ep_2d^2$ for some $\ep_1\le \ep_0$ and $\ep_2>0$. Indeed, instead
of $G$  we might work in a still dense subgraph $H$ of it, having the properties
listed before and by resetting $d:=d(H)\ge d(G)/2$ it suffices to find in $H$ a $TK_{\ell}$ with $\ell=\Omega(d(H))$. The next lemma finds in $G$ a ``nice'' subgraph with ``bounded''
maximum degree.

\begin{lemma}\label{max-deg-reduction}
Let $0<\ep_1<1$ and $\ep_2>0$. Let $G$ be an $n$-vertex bipartite, $C_4$-free, $(\ep_1,\ep_2d^2)$-expander graph with average degree $d$ and $\de(G)\ge d/2$. Then either $G$ contains a subdivision of a clique of order linear in $d$, or $G$ has a $C_4$-free subgraph $G'$ with average degree $d(G')\ge d/2$ and $\de(G')\ge d(G')/4$, that is $(\ep_1/8,4\ep_2d(G')^2)$-expander. Furthermore, $G'$ has at least $n/2$ vertices and $\Delta(G')\le d(G')\log^{8}(|V(G')|/d(G')^2)$.
\end{lemma}

Note that we do not use the $C_6$-freeness of $G$ in Lemma~\ref{max-deg-reduction}. Using Lemma~\ref{max-deg-reduction}, to prove Theorem~\ref{mainc6}, it will be
sufficient to show Theorem~\ref{thm3} below.

\begin{theorem}\label{thm3}
Let $0<\ep_1\le \ep_0$ and $\ep_2>0$, where $\ep_0$ is the constant from Corollary~\ref{k-sz-expander}. Let $G$ be an $n$-vertex bipartite, $\{C_4,C_6\}$-free, $(\ep_1,\ep_2d^2)$-expander graph with average degree $d$, $\de(G)\ge d/4$ and $\Delta(G)\le d\log^8n$. Then $G$
contains a $TK_{\ell/2}$ for $\ell=cd$ for some constant $c>0$ independent of $d$.
\end{theorem}

We will need the following ``independent bounded differences inequality'' (see~\cite{McD}).
\begin{theorem}\label{chernoff}
Let $\mathbf{X}=(X_1,X_2,\ldots,X_n)$ be a family of independent random variables with $X_k$ taking values in a set $A_k$ for each $k$. Suppose that the real-valued function $f$ defined on $\prod A_k$ satisfies $|f(\mathbf{x})-f(\mathbf{x}')|\le \sigma_k$ whenever the vectors $\mathbf{x}$ and $\mathbf{x}'$ differ only in the $k$-th coordinate. Let $\mu$ be the expected value of the random variable $f(\mathbf{X})$. Then for any $t\ge 0$, 
$$\bP(|f(\mathbf{X})-\mu|\ge t)\le 2e^{-2t^2/\sum\sigma_k^2}.$$
\end{theorem}

The rest of the paper will be organized as follows: The proof of
Lemma~\ref{max-deg-reduction} will be given in Section~\ref{Sec-reduction} as
well as the reduction of Theorem~\ref{mainc6} to Theorem~\ref{thm3}. The proof
of Theorem~\ref{thm3} will be divided into two parts according to the range of
$d$: the dense case when $d\ge \log^{14} n$ will be handled in
Section~\ref{sec-dense}, and the sparse case when $d<\log^{14}n$ in
Section~\ref{sec-sparse}. The proof of
Theorem~\ref{c4dense} will be given in Section~\ref{sec-c4dense}. In Section~\ref{cm}, we will give some concluding remarks.


\section{Reduction to ``bounded'' maximum degree}\label{Sec-reduction}
Let $G$ be an $n$-vertex bipartite $C_4$-free $(\ep_1,\ep_2d^2)$-expander graph
with average degree $d$ and $\delta(G)\ge d/2$.

In this section, we will show that we can transform $G$ into a subgraph $G'$
with $d(G')\ge d/2$, $\de(G')\ge d(G')/4$ and $\De(G')\le d(G')\log^8(|V(G')|/d(G')^2)$, where $G'$ is an $(\ep_1/8,4\ep_2d(G')^2)$-expander. For simplicity, throughout this section, define $$t:=\ep_2d^2\quad\quad \mbox{ and } \quad\quad t':=4\ep_2d(G')^2.$$ 
To prove Lemma~\ref{max-deg-reduction}, we shall use the following two lemmas: Lemmas~\ref{upp-max} and~\ref{still-ep}.

Choose a constant $c<\frac{1}{24000}$ such that $c\ll\ep_1$. Set the parameters as follows:
$$\ell=cd,\quad  m=\log\frac{15n}{t}, \quad \Delta=\frac{dm^8}{600},\quad \Delta'=dm^4,\quad \ep(n,\ep_1,t)=\frac{\ep_1}{m^2},\quad diam=\frac{2m^3}{\ep_1}.$$
Note that $d$ has to be sufficiently large (say $d>1/c$) so that $\ell\ge 1$.

If $m\le 1/c^2$, then $d\ge e^{-1/2c^2}n^{1/2}$, and we can apply Theorem~\ref{c4dense} to get a subdivision of a clique of order linear in $d$. Thus we may assume that $1/m\ll c\ll \ep_1$. By the same argument, we may also assume that $d\De\le n$ and $n/d^2\gg 1/\ep_2$.

Let $L\subseteq V(G)$ be the set of all vertices of degree at least $\Delta$.
\begin{lemma}\label{upp-max}
We can find in $G$ either a $TK_{\ell/2}$, or $|L|\le \ell$ and
$G':=G[V\setminus L]$ has maximum degree at most $\Delta$.
\end{lemma}
\begin{proof}
Indeed, if $|L|\ge \ell$, then we can choose a subset $L'\subseteq L$ of exactly $\ell$ vertices, say
$L':=\{v_1,\ldots,v_{\ell}\}$. We shall build a copy of $TK_{\ell/2}$ using a subset of
these high-degree vertices from $L'$ as core vertices. 

First we choose for each vertex $v_i$, $S_1(v_i)\subseteq S(v_i,1)$ and
$S_2(v_i)\subseteq S(v_i,2)$ such that: 

(i) all $S_1(v_i)$'s are pairwise disjoint, and each $S_1(v_i)$ is disjoint from $L'$ and of size $\De/2$;

(ii) every $S_2(v_i)$ is disjoint from $\bigcup_{j=1}^{\ell}S_1(v_j)\cup L'$, and each
$S_2(v_i)$ is of size $d\De/5$;

(iii) for every $1\le i\le \ell$, each vertex in $S_1(v_i)$ has at most $d/2$
neighbors in $S_2(v_i)$.

We can indeed select such sets: 

For (i), since $G$ is $C_4$-free, for any $v_i$, every other $v_j$ with $j\neq
i$ has at most one neighbor in $S(v_i,1)$. Since $|S(v_i,1)|-2\ell\ge
\Delta-2\ell\ge \Delta/2$,
we can remove these neighbors of $v_j$'s and $L'$ from $S(v_i,1)$ and then choose exactly
$\Delta/2$ vertices for $S_1(v_i)$.

For (ii) and (iii), recall that $G$ is bipartite and $\delta(G)\ge d/2$. Thus we can choose, for each
vertex in $S_1(v_i)$, exactly $d/2-1$ vertices in $S(v_i,2)$. Since $G$ is
$C_4$-free, for a given $v_i$, all chosen vertices should be distinct. Thus we have
chosen at least $(d/2-1)(\De/2) \ge 100\ell\De \ge 100
\left|\bigcup_{j=1}^{\ell}S_1(v_j)\right|$ vertices, simply discard those
vertices which are in $\bigcup_{j=1}^{\ell}S_1(v_j)\cup L'$ and then choose $d\De/5$ vertices
for $S_2(v_i)$. Clearly $S_2(v_i)$ satisfies both (ii) and (iii).

\medskip

We now describe the greedy algorithm that we use to connect the vertices
$v_i$'s. Denote by $B_1(v_i):=S_1(v_i)\cup \{v_i\}$ and by $B_2(v_i):=B_1(v_i)\cup S_2(v_i)$.

\medskip

\noindent\textbf{Greedy Algorithm:} We try to connect these $\ell$ core vertices
pair by pair in an arbitrary order. For the current pair of core vertices
$v_i,v_j$, we try to connect $B_2(v_i)$ and $B_2(v_j)$ using a shortest path of
length at most $diam$ and then exclude all the internal vertices in this path
from further connections. We need to justify that such a short path exists.

Suppose we have already connected some pairs using paths of length at most
$diam$. We will exclude all previously used vertices from $B_1(v_i)\cup
B_1(v_j)$ and also those vertices from $S_2(v_i)$, $S_2(v_j)$ adjacent to removed vertices from
$S_1(v_i)$ or $S_1(v_j)$. Formally, let $U$ be the set of vertices used in previous connections and denote by $U_i:=U\cap S_1(v_i)$ and by $U_j:=U\cap S_1(v_j)$. Define
$N:=\left(\Gamma(U_i)\cap S_2(v_i)\right)\cup\left( \Gamma(U_j)\cap
S_2(v_j)\right)$. Then the set of vertices excluded is $U\cup N$. First we bound the size of $U$, it is at most 
$$\ell^2\cdot diam\le c^2d^2 \cdot\frac{2m^3}{\ep_1}\le cd^2m^3,$$ 
as there are at most $\ell^2$ pairs of core vertices
and for each connection, the length of a path is bounded by $diam$.

Call a core vertex $v_i$ bad, if more than $\De'$ vertices from $S_1(v_i)$ are
used in previous connections. During the connections, we discard a core vertex
when it becomes bad. We discard in total at most $\ell/2$ core vertices. Indeed,
we have used at most $\ell^2\cdot diam$ vertices. Since by (i), $S_1(v_i)$'s are pairwise disjoint, each bad core vertex, by
definition, uses at least $\De'$ of them. Thus the number of discarded bad core
vertices is at most
$$\frac{\ell^2\cdot diam}{\Delta'}\le
\frac{cd^2m^3}{dm^4}=\frac{cd}{m}\ll\frac{\ell}{2}.$$
Hence there are at least $\ell/2$ core vertices survive the entire process.

Recall that by (iii), each vertex in $U_i$ (or $U_j$ resp.) has at most
$d/2$ neighbors in $S_2(v_i)$ (or $S_2(v_j)$ resp.). Note that every survived core
vertex is not bad, namely $|U_i|\le \De'$. Thus $|N|\le \De'\cdot
d/2=d^2m^4/2$. Hence the total number of vertices we exclude from $B_2(v_i)$ (or $B_2(v_j)$ resp.) is at
most
$$\ell^2\cdot diam+|N|\le cd^2m^3+\frac{1}{2}d^2m^4\le d^2m^4.$$ 
After excluding these vertices, we still have at least 
$$|S_2(v_i)|-\ell^2\cdot diam-|N|\ge
\frac{d\De}{5}-d^2m^4\ge
\frac{d\De}{10}$$
vertices left in $S_2(v_i)$, the same holds for $S_2(v_j)$. Recall that, when $x\ge
t/2$, $\ep(x,\ep_1,t)$ is decreasing and $x\ep(x,\ep_1,t)$ is increasing. So we have that the number
of vertices we are allowed to exclude, by Corollary~\ref{diameter}, is at least
$$\frac{1}{4}\cdot\frac{d\Delta}{10}\cdot\ep\left(\frac{d\Delta}{10},\ep_1,t\right)\ge\frac{d\Delta}{40}\cdot\ep(n,\ep_1,t)\ge\frac{d^2m^8}{24000}\cdot\frac{\ep_1}
{m^2}=\frac{\ep_1d^2m^6}{24000}\gg d^2m^4,$$
where the last inequality follows from $1/m\ll c\ll \ep_1$ and $c<\frac{1}{24000}$. Thus the exclusion of these vertices will not affect the robust small diameter property between   $B_2(v_i)$'s. So the $\ell/2$ remaining core vertices can be connected to
form a $TK_{\ell/2}$. 
\end{proof}

Given that $c$ is sufficiently small and now we can assume $|L|\le\ell$, we have that $|V(G')|\ge n-\ell\ge n/2$. Note that $d(G')\ge\frac{2(dn/2-\ell n)}{n}=d-2\ell \ge d/2$, thus $t'\ge t$. On the other hand, $G'=G[V\setminus L]$ and $L$ consists of vertices of degree at least $\De\gg d$, thus $d(G')\le \frac{nd-|L|\De/2}{n-|L|}\le d$. Hence $t'\le 4t$ and $\de(G')\ge \de(G)-\ell\ge d/2-\ell\ge d(G')/4$.

\begin{lemma}\label{still-ep}
The obtained graph $G'$ is an $(\ep_1/8,t')$-expander.
\end{lemma}
\begin{proof}
Recall that $t\le t'\le 4t$. Since $G$ is an $(\ep_1,t)$-expander, for any set $X$ in $G'$ of size $x\ge t'/2\ge t/2$, it is easy to check that
\begin{eqnarray*}
|\Gamma_G(X)|&\ge& x\cdot\ep(x,\ep_1,t)=x\cdot\frac{\ep_1}{\log^2(15x/t)}\ge x\cdot \frac{\ep_1/4}{\log^2(15x/t')}= x\cdot\ep(x,\ep_1/4,t')\\
&\ge& \frac{t'}{2}\cdot\ep\left(\frac{t'}{2},\frac{\ep_1}{4},t'\right)=\frac{\ep_1t'}{8\log^2(7.5)}\gg\ell\ge |L|.
\end{eqnarray*}
Hence $|\Gamma_{G'}(X)|\ge |\Gamma_G(X)|-|L|\ge x\ep(x,\ep_1/4,t')-\ell\ge \frac{1}{2}x\ep(x,\ep_1/4,t')=x\ep(x,\ep_1/8,t')$.
\end{proof}

Recall that $1/\ep_2\ll n/d^2\le 2|V(G')|/d(G')^2$, the maximum degree of $G'$ is at most
$$\De=\frac{dm^8}{600}\le \frac{d(G')}{300}\cdot\log^8\frac{30|V(G')|}{\ep_2d(G')^2}\le \frac{d(G')}{300}\left(2\log\frac{|V(G')|}{d(G')^2}\right)^8\le d(G')\log^8\frac{|V(G')|}{d(G')^2}.$$
Slightly abusing the notation, we work in the future only with $G'$. We will
rename $G'$ as $G$, relabelling $n=|V(G')|$ and $d=d(G')$, and by changing $\ep_1$
to $\ep_1/8$ and $\ep_2$ to $4\ep_2$, we assume that $G$ is $(\ep_1,\ep_2d^2)$-expander and its maximum degree is at most $d\log^8(n/d^2).$
This completes the reduction step, i.e.,~to prove
Theorem~\ref{mainc6} it is sufficient to prove Theorem~\ref{thm3}.
\section{Dense case of Theorem~\ref{thm3}}\label{sec-dense}
In this section, we prove the following lemma, which covers the dense case of Theorem~\ref{thm3}. 
\begin{lemma}\label{lem-dense}
Let $0<\ep_1\le \ep_0$ and $\ep_2>0$, where $\ep_0$ is the constant from Corollary~\ref{k-sz-expander}. Let $G$ be an $n$-vertex bipartite, $\{C_4,C_6\}$-free, $(\ep_1,\ep_2d^2)$-expander graph with average degree $d\ge\log^{14}n$, $\de(G)\ge d/4$ and $\Delta(G)\le d\log^8n$. Then $G$ contains a $TK_{\ell/2}$ for $\ell=cd$ for some constant $c>0$ independent of $d$.
\end{lemma}

Let $G$ be a graph satisfying the conditions in Lemma~\ref{lem-dense}.
Choose a constant $c>0$ such that $c\ll \ep_1$ and set $\ell=cd$. In addition, set the parameters in this section as follows:
$$\Delta=d\log^8n, \,\,\,\, \Delta''=d\log^{13}n, \,\,\,\,
b=\frac{d}{\log^9n}, \,\,\,\, diam=\frac{2}{\ep_1}\log^3\left(\frac{15n}{\ep_2d^2}\right)\le\frac{1}{c}\log^3n.$$
Note that $\De\gg d\gg b$, $\Delta''=o(d^2)$, and $\ell/b\le d/b=\log^9n$.

We will first find $\ell$ vertices, $v_1,\ldots,v_\ell$ serving as core
vertices, along with some sets $B_3(v_i)\subseteq B(v_i,3)$. We then connect all
core vertices by linking $B_3(v_i)$'s using a greedy algorithm. Similarly to the proof in Section~\ref{Sec-reduction}, we might discard few core vertices during the process.

\subsection{Choosing core vertices and building $B_3(v_i)$}
We will select $\ell$ vertices $v_1,\ldots,v_{\ell}$ in $\ell/b$ steps to serve as
core vertices.

\medskip

\noindent\textbf{Stage 1:} We choose core vertices $v_1,\ldots, v_\ell$ and the sets
$B_2(v_i)$'s.

\medskip

In each step, we choose a block of vertices consisting of: $b$ core vertices and
for each core vertex $v_i$ a set $B_2(v_i):=S_1(v_i)\cup S_2(v_i)\cup \{v_i\}$, where
$S_1(v_i)\subseteq S(v_i,1)$ and $S_2(v_i)\subseteq S(v_i,2)$ with the following
properties:

(i) $S_1(v_i)$'s are pairwise disjoint for all $1\le i\le \ell$ and
$|S_1(v_i)|=d/2$.

(ii) For every $i$, $|S_2(v_i)|=d^2/10$.

(iii) Every vertex $w\in S_1(v_i)$ has at most $d/4$ neighbors in $S_2(v_i)$.

(iv) Inside each block, the sets $B_2(v_i)$'s are pairwise disjoint.

(v) Every $S_2(v_i)$ is disjoint from $\cup_{j=1}^{\ell} S_1(v_j)$.

(vi) For every $i\neq j$, $v_i\not\in B_2(v_j)$.

To achieve this, we first choose a core vertex $v_i$ with sets $S_1(v_i)$ of size $d/2$ and $S'_2(v_i)\subseteq S(v_i,2)$ of size $d^2/8-d/2$ for all $i\le \ell$. We then choose $S_2(v_i)\subseteq S'_2(v_i)$. Suppose we have chosen some core vertices
$v_1,v_2,\ldots, v_{i-1}$ and sets $S_1(v_j)$ and $S'_2(v_j)$'s for $j\le i-1$. Denote by $D$ the current block and let $B_1(v_j):=S_1(v_j)\cup\{v_j\}$, $j\le i-1$. To choose the next core vertex $v_i$, we will exclude
$\{\bigcup_{j\le i-1} B_1(v_j)\}\cup\{\bigcup_{v_k\in D}S'_2(v_k)\}$. The number
of excluded vertices is at most 
$$\sum_{j\le i}|B_1(v_j)|+b\cdot \max_{v_k\in D}|S'_2(v_k)|\le \ell d+b\cdot
d^2/2\le b\cdot d^2.$$
The number of the edges incident to the excluded vertices is at most 
$$\De\cdot b\cdot d^2 =\frac{d^4}{\log n}\ll \frac{dn}{2}=e(G),$$
the last inequality holds since $G$ is $C_6$-free and therefore $d=O(n^{1/3})$ (see~\cite{B-S}).
Thus, we can easily find in $G$, excluding these vertices, a subgraph $G'$ with
average degree at least $d/2$ and minimum degree at least $d/4$. We then choose
$v_i$ to be any vertex in $G'$ of degree at least $d/2$. Choose $d/2$ neighbors
of $v_i$ to be $S_1(v_i)$. Since $G$ is bipartite, for each vertex $u\in
S_1(v_i)$, we can choose $d/4-1$ neighbors of $u$ not in $B_1(v_i)$. Again, by
$C_4$-freeness, we have chosen $d^2/8-d/2$ different vertices. Denote the resulting set
$S_2'(v_i)$. Note that in the process above, for any $i>j$, the set $S_1(v_i)$ is chosen after $S_2'(v_j)$. Thus when choosing $S_1(v_i)$, vertices in $S_2'(v_j)$ could be included if $v_i$ is in a different block from $v_j$. 
Since $|S_2'(v_i)\setminus \cup_{j\le \ell}S_1(v_j)|\ge |S_2'(v_i)|-\ell\cdot d\ge d^2/10$, we choose a subset of $S_2'(v_i)\setminus \cup_{j\le \ell}S_1(v_j)$ of size exactly $d^2/10$ to be $S_2(v_i)$. 

\medskip

\noindent\textbf{Stage 2:} For each $1\le i\le \ell$, choose $S_3(v_i)$ of size $d^3/50$ and
$B_3(v_i)$.

\medskip

For each vertex in $S_2(v_i)$, since $G$ is bipartite and $C_4$-free, we can
choose $d/4-1$ of its neighbors not in $S_1(v_i)\cup S_2(v_i)$ and denote the
resulting set $S_3'(v_i)$. Since $G$ is $C_6$-free, $|S_3'(v_i)|=|S_2(v_i)|\cdot
(d/4-1)=d^3/40-d^2/10$. Delete from $S_3'(v_i)$ any vertex in $\bigcup_{1\le
j\le \ell} B_1(v_j)$. Since we delete at most $d^2$ vertices, we can choose a
subset of size $d^3/50$ to be $S_3(v_i)$. Let $B_3(v_i):=B_2(v_i)\cup S_3(v_i)$.

\subsection{Connecting core vertices}\label{sec-conn}

\noindent\textbf{Greedy Algorithm:} Now we will connect the $\ell$ core
vertices pair by pair in an arbitrary order. For each pair $v_i$ and $v_j$, we will
connect them with a path of length at most $diam$ avoiding $\bigcup_{p\neq
i,j}B_1(v_p)$.

\medskip

\textbf{(I) Discard bad core vertices:} 

\medskip

Call a core vertex $v_i$ bad, if we use more than $\Delta''$ vertices from
$S_2(v_i)$. Discard a core vertex as soon as it becomes bad. During the entire
process, we use at most $\ell^2\cdot diam$ vertices from previous connections.
Since $B_2(v_i)$'s are pairwise disjoint inside each block, each of the excluded
vertices can appear in at most $\ell/b$ many $S_2(v_i)$'s. Hence, the number of
bad core vertices is at most:
$$\frac{\ell^2\cdot diam\cdot (\ell/b)}{\Delta''}\le\frac{d^2\cdot diam\cdot
(\ell/b)}{d\log^{13}n}\le \frac{d\log^3n\cdot\ell}{cb\log^{13}n}=
\frac{\ell}{c\log n}\ll \ell/2.$$

\medskip

\textbf{(II) Cleaning before connection:} 

\medskip

Assume that we have already connected some pairs of core vertices, and now we want to connect $v_i$ and
$v_j$. Before we start connecting them, clean $B_3(v_i)$ (do the same for
$B_3(v_j)$) in the following way. Notice that we have used in previous
connections at most $\ell$ vertices in $S_1(v_i)$, at most $\De''$ vertices in
$S_2(v_i)$ and at most $\ell^2\cdot diam$ vertices in $S_3(v_i)$, since vertices
in $S_1(v_i)$ were only used when connecting $v_i$ to other core vertices and
$v_i$ is not bad. Also, delete those vertices that are no longer available, i.e.,~those adjacent to used ones. Call the resulting set $B_3'(v_i)$. Since every vertex in
$S_k(v_i)$ for $k\in\{1,2\}$ has at most $d/4$ neighbors in $S_{k+1}(v_i)$, we have
deleted at most $\ell(1+d/4+d^2/16)+\De''(1+d/4)+\ell^2\cdot diam\ll d^3/100$
vertices. Thus $|B_3'(v_i)|\ge |B_3(v_i)|-d^3/100\ge d^3/100$.

\medskip

\textbf{(III) Connecting core vertices:}

\medskip

We will connect $v_i$ and $v_j$ by a shortest path from $B'_3(v_i)$ to
$B_3'(v_j)$ avoiding $\bigcup_{p\neq i,j}B_1(v_p)$ which is of size at most $d^2$. This path has length at most $diam$ if we do not break the robust
diameter property. We then exclude this path for further connections. The number
of excluded vertices from previous paths and from $\bigcup_{p\neq i,j}B_1(v_p)$
is at most $\ell^2\cdot diam+d^2\le d^2\log^3n$. On the other hand, the number
of vertices we are allowed to exclude without breaking the robust small diameter
among $B'_3(v_i)$'s is
$$\frac{1}{4}|B_3'(v_i)|\ep(|B'_3(v_i)|)\ge \frac{d^3}{400}\ep(n)
\ge\frac{\ep_1d^3}{400\log^2n}\gg d^2\log^3n.$$ Thus the robust diameter
property is guaranteed during the entire process.

\medskip

This completes the proof of Lemma~\ref{lem-dense}, hence the dense case of Theorem~\ref{thm3}.

\section{Sparse case of Theorem~\ref{thm3}}\label{sec-sparse}
In this section, we will prove the sparse case of Theorem~\ref{thm3}. Throughout this section $G$ will be a sparse graph satisfying the conditions in Theorem~\ref{thm3}, i.e.,~an $n$-vertex bipartite $\{C_4,C_6\}$-free $(\ep_1,\ep_2d^2)$-expander graph, with average degree $d\le\log^{14}n$, $\de(G)\ge d/4$ and $\Delta(G)\le
d\log^8n$. We always use $n$ for $|V(G)|$ and $d$ for $d(G)$. Inspired by an idea from~\cite{K-O-4} together with a random sparsening trick, we will show that in the sparse case, either we can find in $G$ a $1$-subdivision (i.e.,~each edge is subdivided once) of some graph $H$ with $d(H)=\Omega(d^2)$, or there is a sparse and ``almost regular'' expander subgraph $G_1$ in $G$. In the first case, we apply Theorem~\ref{thm-B-T-K-Sz} to find a subdivision of $K_{\ell}$ in $H$, hence in $G$, with $\ell=\Omega(\sqrt{d(H)})=\Omega(d)$. For the second case, we use the following result of Koml\'os and Szemer\'edi (Theorem 3.1 in~\cite{K-Sz-1}).
\begin{theorem}\label{thm-K-Sz-sparse}
If $F$ is an $(\ep_1,d(F))$-expander satisfying $d(F)/2\le \de(F)\le \De(F)\le 72(d(F))^2$ and $d(F)\le \exp\{(\log|V(F)|)^{1/8}\}$, then $F$ contains a copy of $TK_{\ell}$ with $\ell=\Omega(d(F))$.
\end{theorem}

The following lemma will be useful.
\begin{lemma}\label{lem-hat}
Let $F=(X\cup Y, E)$ be a bipartite $C_4$-free graph. If $|X|=\Omega(d^2|Y|)$ and $\frac{e(F)}{|X|}=\Omega(\De(X))$, then $F$ contains a copy of $TK_{\ell}$ with $\ell=\Omega(d)$.
\end{lemma}
\begin{proof}
In $F$, we call a path of length $2$ with endpoints in $Y$ a \emph{hat}. By the convexity of the function $f(x)={x\choose 2}$, we have that the total number of hats in $F$ is at least 
$$\sum_{v\in X}{\deg(v)\choose 2}\ge \frac{|X|}{3}\cdot \left(\frac{e(F)}{|X|}\right)^2.$$
By the pigeonhole principle, there exists a collection of hats $\cH$ with distinct midpoints of size
$$|\cH|\ge \frac{|X|}{3(\De(X))^2}\cdot \left(\frac{e(F)}{|X|}\right)^2=\Omega(|X|)=\Omega(d^2|Y|).$$

Define a graph $H$ on vertex set $Y$, where two vertices $y,y'\in Y$ are adjacent if there is a hat in $\cH$ with $y,y'$ as endpoints. Note that since $F$ is $C_4$-free, any two hats have different sets of endpoints. Hence, each hat in $\cH$ gives rise to a distinct edge in $H$. Thus 
$$d(H)=\frac{2e(H)}{|Y|}=\frac{2|\cH|}{|Y|}=\Omega(d^2).$$
Since the hats in $\cH$ have distinct midpoints, there is a $1$-subdivision of $H$ in $F$ with core vertices in $Y$ and hats in $\cH$ served as subdivided edges. We then apply Theorem~\ref{thm-B-T-K-Sz} to find a subdivision of $K_{\ell}$ in $H$, hence in $F$, with $\ell=\Omega(\sqrt{d(H)})=\Omega(d)$.
\end{proof}

Let $B:=\{v\in V(G): \deg_G(v)\ge d^3\}$ and $A:=V(G)\setminus B$. Note that $|B|\le \frac{d\cdot |V(G)|}{d^3}=\frac{n}{d^2}$, hence $|A|=|V(G)|-|B|\ge \frac{9n}{10}$. We first show that we may assume that there is a $G'\subseteq G$ with $|V(G')|=\Omega(n)$, $d(G')=\Theta(d)$ and $\De(G')\le d^3$.

\begin{lemma}\label{lem-d3}
We can find in $G$ either a $TK_{\ell}$ with $\ell=\Omega(d)$, or there is a $G'\subseteq G$ with $|V(G')|\ge 9n/10$, $d/20\le d(G')\le d$ and $\De(G')\le d^3$. In the later case, there is a set $A'\subseteq V(G')$ such that $|A'|\ge |V(G')|/2$ and for any $v\in A'$, $\deg_{G'}(v)\ge d/10$.
\end{lemma}
\begin{proof}
Define $G':=G[A]$, $A':=\{v\in A: \deg_{G'}(v)\ge d/10\}$ and $A'':=A\setminus A'$. We distinguish two cases based on the sizes of $A'$ and $A''$.

\medskip

\noindent\textbf{Case 1:} Assume $|A''|\ge |A|/2$. Then $|A''|\ge 9n/20=\Omega(d^2|B|)$. Note that, by the definition of $A''$, for any $a\in A''$, we have $\deg_{G[A'',B]}(a)\ge\de(G)-\deg_{G'}(a)\ge d/4-d/10\ge d/10$. We bound in $G[A'',B]$ the degree of vertices in $A''$ as follows: for each $a\in A''$ with more than $d$ edges to $B$, keep exactly $d$ of them and delete the rest. Let the resulting graph be $G''$. Then in $G''$, $\De(A'')\le d$, hence $\frac{e(G'')}{|A''|}\ge \de(A'')\ge d/10=\Omega(\De(A''))$. Applying Lemma~\ref{lem-hat} to $G''$ gives the first alternative of the conclusion of Lemma~\ref{lem-d3}.

\medskip

\noindent\textbf{Case 2:} Assume $|A'|\ge |A|/2$. The graph $G'$ was obtained from $G$ by removing vertices of degree at least $d^3$ (which were in $B$), thus $d(G')\le d$. On the other hand, by the definition of $A'$, we have $d(G')\ge \frac{|A'|\cdot d/10}{|A|}\ge d/20$ and $\De(G')\le d^3$ as desired.
\end{proof}
From now on, we will work only in $G'=G[A]$ with the properties listed in Lemma~\ref{lem-d3}. For the rest of the proof in this section, we fix sufficiently large constants $C'\ll C\ll K$ and a small constant $c_0\le \frac{1}{1000}$.

Let $W:=\{v\in V(G'): \deg_{G'}(v)\ge c_0d^2\}$, and $U:=V(G')\setminus W$. Note that $|W|\le \frac{d(G')\cdot |V(G')|}{c_0d^2}\le \frac{n}{c_0d}$, hence $|U|=|A|-|W|\ge \frac{4n}{5}$. 

\begin{lemma}\label{lem-d2}
We can find in $G'$ either a $TK_{\ell}$ with $\ell=\Omega(d)$, or there exist vertex sets $U_0\subseteq U$ and $W_0\subseteq W$ with $|U_0|\ge |U|/6$ and $|W_0|\le 2C|W|/d$ such that $G'[U_0,W_0]$ has at least $C'|U_0|$ edges and every vertex in $U_0$ has degree at most $K$ in $G'[U_0,W_0]$.
\end{lemma}

We first show how Lemma~\ref{lem-d2} completes the proof of the sparse case of Theorem~\ref{thm3}. Let $U_0,W_0$ be sets with properties listed in Lemma~\ref{lem-d2}. Note that $|U_0|=\Omega(d^2|W_0|)$. Denote by $G_0:=G'[U_0,W_0]$. Recall that $\De(U_0)=K=O(1)$, thus $\frac{e(G_0)}{|U_0|}\ge C'=\Omega(\De(U_0))$. Applying Lemma~\ref{lem-hat} to $G_0$ gives a copy of $TK_{\ell}$ with $\ell=\Omega(d)$. This completes the proof of the sparse case of Theorem~\ref{thm3}.

\begin{proof}[Proof of Lemma~\ref{lem-d2}]
Recall that $A'\subseteq V(G')$ consists of vertices of degree at least $d/10$ in $G'$. Define $U':=\{v\in A'\cap U: \deg_{G'[U,W]}(v)\ge d/20\}$ and
$U'':=\{A'\cap U\}\setminus U'$. By Lemma~\ref{lem-d3}, $|A'|\ge \frac{|V(G')|}{2}=\frac{|U|+|W|}{2}$. Thus $|U'|+|U''|=|A'\cap U|\ge |A'|-|W|\ge \frac{|U|-|W|}{2}\ge\frac{2|U|}{5}$.
We distinguish two cases based on the sizes of $U'$ and $U''$.

\medskip

\noindent\textbf{Case 1:} $|U''|\ge |U|/5$. Note that for every $v\in U''$, by the definition of $U''$,
$$\deg_{G'[U]}(v)=\deg_{G'}(v)-\deg_{G'[U,W]}(v)\ge\frac{d}{10}-\frac{d}{20}=\frac{d}{20}.$$
Thus $d(G'[U])\ge \frac{d/20\cdot |U''|}{|U|}\ge d/100$ and by the definition of $U$ we have $\De(G'[U])\le c_0d^2$. Then we apply Corollary~\ref{k-sz-expander} to $G'[U]$ and let $G_1$ be the resulting $(\ep_1, \ep_2d(G_1)^2)$-expander subgraph with $\ep_2<1/1000$, $d(G_1)\ge d(G'[U])/2\ge d/200$, $\de(G_1)\ge d(G_1)/2$ and $\De(G_1)\le\De(G'[U])\le c_0d^2$. Let $n_1:=|V(G_1)|$. 

If $d(G_1)\ge \exp\{(\log n_1)^{1/8}\}$, then we apply Lemma~\ref{max-deg-reduction} to $G_1$. Then either we have a copy of $TK_{\ell}$ with $\ell=\Omega(d)$, in which case we are done, or we obtain a subgraph $G_2\subseteq G_1$ with $d(G_2)\ge d(G_1)/2\ge d/400$, $\de(G_2)\ge d(G_2)/4$ and $\De(G_2)\le d(G_2)\log^8\frac{|V(G_2)|}{d(G_2)^2}$, which is an $(\ep_1/8, 4\ep_2d(G_2)^2)$-expander. Since $|V(G_2)|\le n_1$, we have that $d(G_2)\ge d(G_1)/2\gg \log^{14}|V(G_2)|$. Applying Lemma~\ref{lem-dense} to $G_2$ gives a $TK_{\ell}$ with $\ell=\Omega(d(G_2))=\Omega(d)$.

We may now assume that $d(G_1)\le \exp\{(\log n_1)^{1/8}\}$. We want to apply Theorem~\ref{thm-K-Sz-sparse} to $G_1$ to get a $TK_{\ell}$ with $\ell=\Omega(d(G_1))=\Omega(d)$. Recall that $d(G_1)/2\le\de(G_1)\le\De(G_1)\le c_0d^2\le 72d(G_1)^2$, where the last inequality follows from $d(G_1)\ge d/200$ and $c_0\le 1/1000$. It suffices to check that $G_1$ is an $(\ep_1, d(G_1))$-expander.
\begin{claim}
The graph $G_1$ is an $(\ep_1, d(G_1))$-expander.
\end{claim}
\begin{proof}
Recall that $G_1$ is bipartite, $C_4$-free and $(\ep_1,\ep_2d(G_1)^2)$-expander. For any set $X$ of size $x\ge \ep_2d(G_1)^2/2$, $|\Gamma(X)|\ge x\cdot\ep(x,\ep_1,\ep_2d(G_1)^2)\ge x\cdot \ep(x,\ep_1,d(G_1))$, as $\ep(x,\ep_1,t)$ is an increasing function in $t$.

It is known that in $C_4$-free bipartite graphs of minimum degree $k$, any set of size at most $k^2/500$ expands by a rate of at least $2$ (see e.g.~Lemma~2.1 in~\cite{S-V}). Recall that $\de(G_1)\ge d(G_1)/2$ and $\ep_2\le 1/1000$, so $\ep_2d(G_1)^2/2\le 2\ep_2\de(G_1)^2\le\frac{\de(G_1)^2}{500}$. Since $\ep(x,\ep_1,d(G_1))$ is a decreasing function in $x$, for any $x\ge d(G_1)/2$, $\ep(x,\ep_1,d(G_1))\le\ep(d(G_1)/2,\ep_1,d(G_1))=\frac{\ep_1}{\log^2(7.5)}<2$. Thus for any set $X$ of size $d(G_1)/2\le x\le \ep_2d(G_1)^2/2\le \frac{\de(G_1)^2}{500}$, we have $|\Gamma(X)|\ge 2x\ge x\cdot\ep(x,\ep_1,d(G_1))$ as desired.
\end{proof}
This gives the first alternative of the conclusion of Lemma~\ref{lem-d2}.

\medskip

\noindent\textbf{Case 2:} $|U'|\ge |U|/5\ge \frac{4n/5}{5}\ge n/7$. Recall that $|W|\le \frac{n}{c_0d}$. Consider the subgraph $G_3:=G'[U',W]$, by deleting extra edges, we may assume that each vertex in $U'$ has degree at most $d$ in $W$. Then by the definition of $U'$, we have 
$$\frac{d}{11}\le \frac{2|U'|\cdot d/20}{|U'|+|W|}\le d(G_3)\le \frac{2|U'|\cdot d}{|U'|+|W|}\le 2d.$$
Set $p:=C/d$. We will choose a random subset $W_0\subseteq W$, in which each element of $W$ is included with probability $p$ independent of each other. We then choose some $U_0\subseteq U'$ consisting of vertices of degree at most $K$ in $W_0$. We will show that with positive probability, $W_0$ and $U_0$ have the desired properties. For simplicity, we define $G_4:=G_3[U',W_0]$.

We may assume that $|W|\ge \frac{n}{d^2}$, since otherwise $|U'|=\Omega(d^2|W|)$ and $\frac{e(G_3)}{|U'|}\ge \de(U')\ge d/20=\Omega(\De(U'))$. Then applying Lemma~\ref{lem-hat} to $G_3$ yields a $TK_{\ell}$ with $\ell=\Omega(d)$. Note that $\bE|W_0|=p|W|$, by Chernoff's Inequality, w.h.p.~$|W_0|\le 2\bE|W_0|=2C|W|/d$. As mentioned above, we will delete vertices from $U'$ with degree more than $K$ in $W_0$ to form $U_0$. It suffices to show that w.h.p. 

(i) $e(G_4)\ge 2C'|U'|$;

(ii) the number of vertices deleted (i.e.,~$U'\setminus U_0$) is at most $|U'|/10$ and the number of edges deleted (from $G_4$ to form $G_3[U_0,W_0]=G'[U_0,W_0]$) is at most $C'|U'|$.

It then follows that $|U_0|\ge 9|U'|/10\ge |U|/6$ and the number of edges in $G_0=G'[U_0,W_0]$ is at least $e(G_4)-C'|U'|\ge C'|U'|\ge C'|U_0|$ as desired.

For (i), recall that by Lemma~\ref{lem-d3}, $\De(G_3)\le d^3$. For each vertex $v_i\in W$, define a random variable $X_i$ taking value $\deg_{G_3}(v_i)$ if $v_i\in W_0$ and $0$ otherwise. Then $e(G_4)=\sum_{i\le |W|} X_i$ and
$$\bE(e(G_4))=\sum_{i\le |W|}\bE X_i=\sum_{v_i\in W}p\cdot\deg_{G_3}(v_i)=p\cdot e(G_3)\ge \frac{C}{d}\cdot \frac{d}{20}\cdot |U'|\ge 4C'|U'|.$$ 
Recall that $\frac{n}{d^2}\le |W|\le \frac{n}{c_0d}$ and $d\le\log^{14}n$. Applying Theorem~\ref{chernoff} with $f(\mathbf{X})=\sum X_i$, $\sigma_i=d^3$ and $t=\bE(e(G_4))/2\ge 2C'|U'|\ge \frac{2C'n}{7}\ge \frac{2C'}{7}\cdot c_0d|W|\ge c_0d|W|$, we have that 
\begin{eqnarray*}
\bP\left[e(G_4)\le \frac{1}{2}\bE(e(G_4))\right]\le 2e^{-\frac{2(c_0d|W|)^2}{d^6|W|}}=e^{-c^2_0|W|/d^4}\le e^{-c^2_0n/d^6}\rightarrow 0.
\end{eqnarray*}

For (ii), for each $u_i\in U'$, we define a random variable $Y_i:=\deg_{G_4}(u_i)$. Note that for any two vertices $u_i, u_j\in U'$, if they have no common neighbor in $W$, then $Y_i$ and $Y_j$ are independent. Define an auxiliary dependency graph $F$ on vertex set $\{Y_i\}_{i=1}^{|U'|}$, in which $Y_i$ and $Y_j$ are adjacent if and only if they are not independent. Since in $G_3$ every vertex in $U'$ has degree at most $d$ and every vertex in $W$ has degree at most $d^3$, it follows that $\De(F)\le d^4$ and by Brook's theorem that $\chi(F)\le d^4+1$. Thus we can partition $U'$ into $d^4+1$ classes, say $U':=Z_0\cup Z_1\cup \ldots\cup Z_{d^4}$, such that $Y_i$'s corresponding to vertices in the same class are independent. First we discard classes of size smaller than $n/d^6$, the number of vertices we delete at this step is at most $\frac{n}{d^6}\cdot (d^4+1)\ll |U'|$. Thus we may assume that each class is of size at least $n/d^6$.
Fix a class $Z_j$, for every $v\in Z_j$ and every $i\ge K\gg C$, 
$$\bP[\deg_{G_4}(v)=i]={\deg_{G_3}(v)\choose i}p^i(1-p)^{\deg_{G_3}(v)-i}\le \frac{d^i}{i!}\cdot \frac{C^i}{d^i}\le e^{-i\log i/2}:=q_i.$$
For each $1\le i\le d$, let $N_i$ ($N_{\ge i}$ resp.) be the number of vertices in $Z_j$ of degree $i$ (at least $i$ resp.) in $W_0$. Then $\bE N_i\le |Z_j|q_i$. For each $i\le\log^{2}d$, by Chernoff's Inequality and recall that $d\le\log^{14}n$, we have
\begin{eqnarray}\label{eq-sp2}
\bP[N_i\ge 2\bE N_i]<\exp\{-|Z_j|q_i/3\}\ll \exp\left\{-\frac{n}{d^6}\cdot e^{-\log^3d}\right\}\ll \exp \left\{-\frac{n}{e^{(\log\log n)^4}}\right\}.
\end{eqnarray}
Note that for any $v\in Z_j$, $\bP[\deg_{G_4}(v)\ge \log^2d]\le \sum_{i=\log^2d}^{d}q_i\ll e^{-\log^2d}$. It follows that 
\begin{eqnarray}\label{eq-sp3}
\bP[N_{\ge\log^2d}\ge 2\bE N_{\ge\log^2d}]\ll\exp\left\{-|Z_j|\cdot e^{-\log^2d}\right\}\ll \exp\left\{-\frac{n}{e^{(\log\log n)^3}}\right\}.
\end{eqnarray}
By~\eqref{eq-sp2},~\eqref{eq-sp3} and the union bound, the probability that there exists a class $Z_j$ in which either $N_{\ge\log^2d}\ge 2\bE N_{\ge\log^2d}$ or for some $i\le\log^2d$, $N_i\ge2\bE N_i$ is at most
$$(d^4+1)\cdot \left(\log^2d\cdot\bP[N_i\ge 2\bE N_i]+\bP[N_{\ge\log^2d}\ge 2\bE N_{\ge\log^2d}]\right)\rightarrow 0.$$
Note that $\sum_{K\le i\le\log^2d}\bE N_i\le\sum_{K\le i\le\log^2d}q_i|Z_j|\ll e^{-K}|Z_j|$.
Thus w.h.p.~the number of vertices deleted is at most 
$$\sum_j((2\sum_{K\le i\le\log^2d}\bE N_i+2\bE N_{\ge\log^2d})\cdot |Z_j|)\ll\sum_j (e^{-K}+e^{-\log^2d})\cdot |Z_j|<2e^{-K}|U'|\ll |U'|.$$
The number of edges incident to vertices deleted in $Z_j$ is at most 
$$\sum_{K\le i\le\log^2d}(2q_i|Z_j|\cdot i)+(\sum_{i=\log^2d}^{d}2q_i|Z_j|)\cdot d\ll(e^{-K}+d\cdot e^{-\log^2d})\cdot |Z_j|<2e^{-K}|Z_j|.$$
Recall that every vertex in $U'$ has degree at most $d$ in $W$ and that $|U'|\ge n/7$. Then summing over all classes, the total number of edges deleted is at most 
$$\sum_{|Z_j|\ge n/d^6}2e^{-K}|Z_j|+\sum_{|Z_k|\le n/d^6}d\cdot |Z_k|\le 2e^{-K}|U'|+(d^4+1)\cdot d\cdot\frac{n}{d^6}\ll |U'|.$$
\end{proof}


\section{Proof of Theorem~\ref{c4dense}}\label{sec-c4dense}
In this section, we will prove Theorem~\ref{c4dense} using a variation of the
Dependent Random Choice Lemma (see survey~\cite{F-S} for more details on
the method of dependent random choice). The following lemma roughly says that in a dense
$C_4$-free graph one can find a set in which every small subset has a large
second common neighborhood.
\begin{lemma}\label{drc}
Let $G=(A\cup B,E)$ be a $C_4$-free bipartite graph on $n$ vertices with
$cn^{3/2}$ edges and $|A|=|B|=\frac{n}{2}$, where $n>1/c^{20}$. If there exist
positive integers $a$, $m$, $r$ and $t$ such that
\begin{eqnarray}\label{eq-drc}
c^{2t}n-{n\choose r}\left(\frac{m}{n/2}\right)^t\ge a,
\end{eqnarray}
then there exists $U\subseteq A$ with at least $a$ vertices such that for every
$r$-subset $S\subseteq U$, $|N_2(S)|\ge m$.
\end{lemma}

\begin{proof}
First notice that
\begin{eqnarray*}
\sum_{v\in A}|N_2(v)|&=&\sum_{v\in B}(d(v)-1)d(v)=\sum_{v\in B}d(v)^2-\sum_{v\in
B}d(v)\ge \frac{n}{2}\left(\frac{\sum_{v\in B}d(v)}{n/2}\right)^2-e(G)\\
&=&\frac{n}{2}(2cn^{1/2})^2-cn^{3/2}\ge c^2n^2.
\end{eqnarray*}
Pick a set $T\subseteq A$ of $t$ vertices uniformly at random with repetition.
Let $W:=N_2(T)\subseteq A$ and put $X:=|W|$. Then by the linearity of
expectation and $t\ge 1$, we have
\begin{eqnarray*}
\mathbb{E}[X]&=&\sum_{v\in A}\mathbb{P}(v\in N_2(T))=\sum_{v\in
A}\left(\frac{|N_2(v)|}{n/2}\right)^t=
\left(\frac{2}{n}\right)^{t}\cdot\frac{n}{2}\cdot\left(\frac{1}{n/2}\sum_{v\in
A}|N_2(v)|^t\right)\\&\ge&
\left(\frac{n}{2}\right)^{1-t}\cdot\left(\frac{\sum_{v\in A} |N_2(v)|
}{n/2}\right)^t
\ge\left(\frac{n}{2}\right)^{1-t}\cdot(2c^2n)^t=2^{2t-1}c^{2t}n\ge c^{2t}n.
\end{eqnarray*}
Let $Y$ be the random variable counting the number of $r$-sets in $W$ that have
 fewer than $m$ common second neighbors. The probability for a fixed such $r$-set $S$
to be in $W$ is at most $\left(\frac{m}{n/2}\right)^t$. There are at most
${n\choose r}$ $r$-sets, hence
$$\mathbb{E}[X-Y]\ge c^{2t}n-{n\choose r}\left(\frac{m}{n/2}\right)^t\ge a.$$
Thus there exists a choice of $T$, such that $X-Y\ge a$. Delete one vertex from
$X$ for each such ``bad'' $r$-set from $W$, and the resulting set $U$ has the
desired property.
\end{proof}

\begin{claim}\label{nice}
When proving Theorem~\ref{c4dense}, we may assume that $G$
is bipartite on $A\cup B$ with $|A|=|B|=n/2$, $d(G)=d$ and all vertices in $B$
have degree smaller than $30d$.
\end{claim}
\begin{proof}
We may assume that for any
$H\subseteq G$, $d(H)\le d$, otherwise we can work in $H$ instead.
Let $X\subseteq V$ be the set of vertices of degree at least $10d$, thus $|X|\le
n/10$. Let $Y=V\setminus X$. Since $d(G[X])\le d$, we have $e(G[X])\le d|X|/2\le
e(G)/10$. Take an $\frac{n}{2}$-subset $B$ of $Y$ uniformly at random and call
$V\setminus B=A$. Then we have,
$$\mathbb{E}(e(G[A,B]))\ge 0.4[e(G[Y])+e(G[X,Y])]=0.4[e(G)-e(G[X])]\ge 0.36
e(G).$$
Therefore there exists a choice of $A,B$ such that $e(G[A,B])\ge 0.36e(G)$.
Hence we can replace $G$ by $G':=G[A,B]$, and every vertex in $B$ has degree
less than $10d\le 10\cdot (d(G')/0.36)<30d(G')$.
\end{proof}

\begin{proof}[Proof of Theorem~\ref{c4dense}]

Assume $G$ satisfies the conditions of Claim~\ref{nice} and apply
Lemma~\ref{drc} to $G$ with the following parameters: 
$$a=\frac{c^6n^{1/2}}{240},\quad r=2,\quad t=\frac{\log n}{4\log(1/c)},\quad m=\frac{c^6n}{2}.$$ 

In order to prove that \eqref{eq-drc} is satisfied, we shall prove $2{n\choose
2}\left(\frac{m}{n/2}\right)^t\le c^{2t}n$ and $c^{2t}n\ge 2a$. Indeed, 
\begin{eqnarray*}
2{n\choose 2}\left(\frac{m}{n/2}\right)^t\le c^{2t}n&\Leftarrow&  n\le
\left(\frac{c^2n/2}{m}\right)^t=\left(\frac{1}{c}\right)^{4t} \Leftarrow \log
n\le 4t\cdot \log\frac{1}{c}=\log n.
\end{eqnarray*}
On the other hand, we have
\begin{eqnarray*}
c^{2t}n\ge 2a=\frac{c^6n^{1/2}}{120} \Leftrightarrow
\frac{120n^{1/2}}{c^6}\ge\left(\frac{1}{c}\right)^{2t} \Leftarrow \log 120+
\frac{1}{2}\log n + 6\log\frac{1}{c}\ge 2t\log\frac{1}{c}=\frac{1}{2}\log n.
\end{eqnarray*}
Thus there exists $U\subseteq A$ of size at least $a=\frac{c^6n^{1/2}}{240}$ such that for
every pair of vertices $S\subseteq U$, $|N_2(S)|\ge m=c^6n/2$.

We embed a copy of $TK_\ell$ with $\ell=a=c^5d/480$ greedily as
follows: first embed all the core vertices arbitrarily to $U$. Then we connect
all pairs of core vertices one by one, in an arbitrary order, with internally
vertex-disjoint paths of length 4. Fix a pair of vertices $S\subseteq U$. For
every vertex $v$ in $N_2(S)$, call $C(v):=N(v)\cap \Gamma(S)$ its
\emph{connector set} and call $v$ ``bad'' if $|C(v)|=1$. 
Since $G$ is $C_4$-free, $|N_1(S)|\le 1$, so there are at most $\De(B)\le 30d$ bad
vertices in $N_2(S)$. Any vertex $v\in N_2(S)$ that is not bad has
$|C(v)|=2$. When connecting $S$, we will exclude from $N_2(S)$ the following vertices: (i) bad vertices (if they exist); (ii) vertices in $U$; (iii) vertices that were already
used in previous connections; (iv) vertices whose connector set was used. It follows
immediately that if there is a vertex left in $N_2(S)$, then together with its
connector set, we can connect $S$.

For (i) and (ii), recall that there are at most $30d$ bad vertices and $|U|\le \ell$.
For (iii), there are at most ${\ell\choose 2}$ such vertices, one for each pair of core vertices. Thus there are at least $m-30d-\ell-{\ell\choose 2}\ge c^6n/2-60cn^{1/2}-\ell^2\ge
c^6n/4$ many vertices left in $N_2(S)$.

For (iv), we say that two vertices in $N_2(S)$ have no \emph{conflict} with each other if their
connector sets are disjoint. Notice that every vertex $v$ in $N_2(S)$ that is not bad can
have a conflict with at most $|C(v)|\cdot\Delta(B)= 2\De(B)\le 60d$ vertices. Thus
we can find at least
$$\frac{c^6n/4}{2\De(B)}\ge
\frac{c^6n}{240d}=\frac{c^6n}{480cn^{1/2}}=\frac{c^5n^{1/2}}{480}\ge 2\ell$$
not-previously-used vertices in $N_2(S)$ that are pairwise
conflict-free. Again since $G$ is $C_4$-free, any other core vertex in
$U\setminus S$ can be adjacent to connector sets of at most 2 vertices in
$N_2(S)$. Thus there are at least $2\ell-2(\ell-2)=4$ vertices available in
$N_2(S)$ to connect the pair of vertices in $S$. 
\end{proof}


\section{Concluding Remarks}\label{cm}
The proof of Theorem~\ref{mainc2k} is almost identical to the proof of
Theorem~\ref{mainc6}. The only differences is to generalize Lemma~\ref{lem-dense} to $\{C_4,C_{2k}\}$-free graphs for any $k\ge 4$. First we need a
result of K\"{u}hn and Osthus~\cite{K-O-3}, which finds a $C_4$-free subgraph
$G'$ in a $C_{2k}$-free graph $G$ for $k\ge 4$ such that $d(G')=\Omega(d(G))$. 
Then after cleaning $S_1(v_i)$ and $S_2(v_i)$(as in Section~\ref{sec-conn}), $S_2(v_i)$ still has $\Omega(d^2)$ vertices. Recall that each vertex in $S_2(v_i)$ sends $\Omega(d)$ edges to $S_3(v_i)$, then by a well-known result of Bondy and Simonovits~\cite{B-S}, we have that there are at least $\Omega(d^{3-3/(k+1)})$ vertices available in $S_3(v_i)$ after cleaning $S_1(v_i)$ and $S_2(v_i)$. We further clean $S_3(v_i)$ by deleting at most $\ell^2\cdot diam$ vertices. For $k\ge 4$, $d^{3-3/(k+1)}\ep(d^{3-3/(k+1)})\gg \ell^2\cdot diam+d^2$, thus the robust
diameter property is guaranteed for all connections.


\section*{Acknowledgement}
We would like to thank Sasha Kostochka for stimulating discussions and helpful comments. We would also like to thank the referees for a careful reading.
%


\begin{thebibliography}{99}
\bibitem{A-K-S}
N.~Alon, M.~Krivelevich and B.~Sudakov,
\newblock Tur\'an numbers of bipartite graphs and related {R}amsey-type questions.
\newblock {\em  Combin. Probab. Comput.}, 12, (2003), 477--494.

\bibitem{B-S}
J.~Bondy, M.~Simonovits,
\newblock Cycles of even length in graphs.
\newblock {\em  J. Combin. Theory Ser. B}, 16, (1974), 97--105.

\bibitem{B-Th}
B.~Bollob{\'a}s, A.~Thomason,
\newblock Proof of a conjecture of {M}ader, {E}rd{\H o}s and {H}ajnal on
topological complete subgraphs.
\newblock {\em  European J. Combin.}, 19, (1998), 883--887.

\bibitem{E}
P.~Erd\H{o}s,
\newblock Problems and results in graph theory and combinatorial analysis, in:
\newblock {\em  Graph theory and
related topics}, (Proc. Conf. Waterloo, 1977), Academic Press, New York (1979), 153–163.

\bibitem{F-S}
J. Fox, B. Sudakov,
\newblock Dependent random choice.
\newblock {\em  Random Structures $\&$ Algorithms}, 38, (2011), 68--99.

\bibitem{G}
E. Gy{\"o}ri,
\newblock {$C_6$}-free bipartite graphs and product representation of squares.
\newblock {\em  Discrete Math.}, 165/166, (1997), 371--375.

\bibitem{J}
H.~A.~Jung,
\newblock Eine {V}erallgemeinerung des {$n$}-fachen {Z}usammenhangs f\"ur {G}raphen.
\newblock {\em  Math. Ann.}, 187, (1970), 95--103.

\bibitem{K-O-1}
D.~K{\"u}hn, D.~Osthus,
\newblock Topological minors in graphs of large girth.
\newblock {\em  J. Combin. Theory Ser. B}, 86, (2002), 364--380.

\bibitem{K-O-2}
D.~K{\"u}hn, D.~Osthus,
\newblock Large topological cliques in graphs without a 4-cycle.
\newblock {\em  Combin. Probab. Comput.}, 13, (2004), 93--102.

\bibitem{K-O-3}
D.~K{\"u}hn, D.~Osthus,
\newblock Four-cycles in graphs without a given even cycle.
\newblock {\em J. Graph Theory}, 48, (2005), 147--156.

\bibitem{K-O-4}
D.~K{\"u}hn, D.~Osthus,
\newblock Induced subdivisions in {$K_{s,s}$}-free graphs of large average degree.
\newblock {\em Combinatorica}, 24, (2004), 287--304.

\bibitem{K-S-T}
T.~K{\"o}v\'ari, V.~T.~S{\'o}s and P.~Tur{\'a}n,
\newblock On a problem of {K}. {Z}arankiewicz.
\newblock {\em Colloquium Math.}, 3, (1954), 50--57.

\bibitem{K-Sz-1}
J.~Koml{\'o}s, E.~Szemer{\'e}di,
\newblock Topological cliques in graphs.
\newblock {\em Combin. Probab. Comput.}, 3, (1994), 247--256.

\bibitem{K-Sz-2}
J.~Koml{\'o}s, E.~Szemer{\'e}di,
\newblock Topological cliques in graphs {II}.
\newblock {\em Combin. Probab. Comput.}, 5, (1996), 79--90.

\bibitem{M}
W.~Mader,
\newblock An extremal problem for subdivisions of {$K^-_5$}.
\newblock {\em  J. Graph Theory}, 30, (1999), 261--276.

\bibitem{McD}
C.~McDiarmid,
\newblock On the method of bounded differences.
\newblock {\em  Surveys in combinatorics, 1989 ({N}orwich, 1989), London Math. Soc. Lecture Note Ser.}, 141, (1989), 148--188.

\bibitem{S-V}
B.~Sudakov, J.~Verstra{\"e}te,
\newblock Cycle lengths in sparse graphs.
\newblock {\em Combinatorica}, 28, (2008), 357--372.
\end{thebibliography}
\end{document}